\documentclass{proc-l}

\usepackage{amssymb}

\usepackage{braket}
\usepackage[colorlinks]{hyperref}

\newtheorem{theorem}{Theorem}[section]
\newtheorem{lemma}[theorem]{Lemma}
\newtheorem{fact}[theorem]{Fact}

\theoremstyle{definition}
\newtheorem{definition}[theorem]{Definition}

\theoremstyle{remark}
\newtheorem{remark}[theorem]{Remark}

\numberwithin{equation}{section}

\def\bbar{\bar{b}}
\def\cbar{\bar{c}}
\def\dbar{\bar{d}}

\def\hbar{\bar{h}}

\def\mbar{\bar{m}}

\def\xbar{\bar{x}}
\def\ybar{\bar{y}}
\def\zbar{\bar{z}}

\def\QF{{\rm QF}}
\def\Fm{{\rm FO}}
    \def\rtp{{\rm rtp}}
    \newcommand{\tp}{\mathrm{tp}}
    \newcommand\acl{\hbox{\rm acl}}
    \def\cc{{\mathbf c}}
    \def\dd{{\mathbf d}}
    
    \renewcommand\>{\rangle}
    \def\Ind#1#2{#1\setbox0=\hbox{$#1x$}\kern\wd0\hbox to 0pt{\hss$#1\mid$\hss}
    	\lower.9\ht0\hbox to 0pt{\hss$#1\smile$\hss}\kern\wd0}
     \def\ind{\mathop{\mathpalette\Ind{}}}
     \newcommand{\Ra}{\Rightarrow}
     \newcommand{\bs}{\backslash}

\begin{document}

\title[Theories with few non-algebraic types over models]	{Theories with few non-algebraic types over models, and their decompositions}


%    author one information
\author{Samuel Braunfeld}
\address{Charles University, Faculty of Mathematics and Physics, Computer Science Institute, Praha, Czechia 11800}

%    author two information
\author[Michael C. Laskowski]{Michael C. Laskowski$^*$}
\address{University of Maryland College Park, Deparatment of Mathematics, College Park, MD 20742}
\thanks{$^*$Partially supported
	by NSF grant DMS-1855789}

%    \subjclass is required.

\date{}

\dedicatory{}

%    "Communicated by" -- provide editor's name; required.
\commby{Vera Fischer}

%    Abstract is required.
\begin{abstract}
	 We consider several ways of decomposing models into parts of bounded size forming a congruence over a base, and show that admitting any such decomposition is equivalent to mutual algebraicity at the level of theories. We also show that a theory $T$ is mutually algebraic if and only if there is a uniform bound on the number of coordinate-wise non-algebraic types over every model, regardless of its cardinality.
\end{abstract}

\maketitle

\section{Introduction}
A key theme in model theory is to identify which theories have models that admit a structure theory, in the sense that their models can be decomposed into simple pieces that relate to each other in a controlled way. Intertwined with this is the theme of determining the complexity of theories by counting the number of types over models. The archetypal example of a structure theory is for classifiable theories, whose models are determined by a well-founded tree of countable elementary substructures \cite{Bus, HHL}. Integral to the analysis of classifiable theories are the properties of stability and superstability, both initially defined by type-counting. 

Here we investigate a family of much stronger decompositions for models, and show that they are all equivalent at the level of theories. In particular, for each type of decomposition, the property that all models of $T$ admit such a decomposition is equivalent to $T$ being {\em mutually algebraic}. (Mutual algebraicity is a condition generalizing bounded-degree graphs, and already has several characterizations \cite{JSL}. For this note, all we need is contained in Facts \ref{fact:exp} and \ref{fact: 4.4}.) Type-counting plays a fundamental role in the proof, and we almost simultaneously obtain a characterization of mutual algebraicity in terms of a very strong type-counting condition.

Stability in a cardinal $\kappa$ is defined by there being only $\kappa$ consistent types over every model $M$ of size $\kappa$. This is as low as possible since for every $m \in M$, there exists the algebraic type containing $x=m$. But by only considering types $p(\xbar)$ that are {\em coordinate-wise non-algebraic} over $M$, i.e. with no variable in $\xbar$ set equal to an element of $M$, we may do better for some theories. Following \cite[Corollary 6.1.8]{BS}, we call a theory {\em bounded} if there is a uniform bound on the number of coordinate-wise non-algebraic types over every model, regardless of its size. We show a theory is bounded if and only if it is mutually algebraic, and in fact the bound on the number of coordinate-wise non-algebraic types is $2^{|T|}$.

Our decomposition conditions are more involved, but they are also rooted in ideas from \cite{BS}.  There, it is shown for a monadically stable theory $T$, i.e. every expansion of $T$ by unary predicates remains stable, every model of $T$ admits a decomposition into an independent tree of countable elementary submodels, as in the classifiable case but without any need to complete to prime models. Our notion of decomposition (in particular, what we call a $(|T|,\QF)$-model decomposition) corresponds to such a tree-decomposition of depth one.

Thus a corollary of our result on decompositions is that the monadically stable theories of depth one are the same as the mutually algebraic theories. This generalizes the corresponding result for the $\omega$-categorical case, which follows from the classification of $\omega$-categorical monadically stable theories in \cite{Lach} and the characterization of $\omega$-categorical mutually algebraic theories in \cite{cell}. 

\section{Preliminaries}

Rather than work with the coordinate-wise non-algebraic types from the introduction,  we recall a notion of complexity $\rtp_\Delta(N,B)$ that was used by the authors in \cite{MonNIP}, which counts the number of $\Delta$-types over $B$ that are realized in $(N-B)^{<\omega}$. When $B$ is a model and $N$ is $|B|^+$-saturated, then $\rtp_\Delta(N,B)$ counts the number of consistent coordinate-wise non-algebraic $\Delta$-types over $B$.

\begin{definition} \label{def:reas} For a fixed language $L$, a set $\Delta$
	of $L$-formulas is {\em reasonable} if it contain all quantifier free formulas and is closed under permutation of variables and boolean combinations.  Examples include the set of quantifier-free formulas $\QF$, boolean combinations of $\Sigma_n$, or the set of all $L$-formulas $\Fm$.

	For an $L$-structure $N$ and a subset $B\subseteq N$, and $\cbar\in (N-B)^k$, let
	$$\tp_\Delta(\cbar/B)=\{\phi(\xbar,\bbar):\phi(\zbar)\in\Delta, \xbar\ybar \text{ a partition of $\zbar$}, \bbar\in B^{\lg(\ybar)}, N\models\phi(\cbar,\bbar)\}$$
	and let $\rtp_{\Delta}(N,B)$ denote the number of $\Delta$-types over $B$ realized in $(N-B)^{<\omega}$.
	When $\Delta=\Fm$, we simply write $\rtp(N,B)$.
\end{definition}

We record the following facts about $\rtp(N,B)$.

\begin{fact} \label{rtpfacts}   Let $B\subseteq N$ be arbitrary.
	\begin{enumerate}
		\item If $\QF\subseteq\Delta\subseteq\Fm$, then $\rtp_{\QF}(N,B)\le \rtp_{\Delta}(N,B)\le \rtp_{\Fm}(N,B)$;
		\item  $\rtp(N,B)\le \beth_{\omega+1}(\rtp_{\QF}(N,B))$.
		\item  Let $L^+$ be an expansion of $L$ by finitely many unary predicates, and $N^+$ a corresponding expansion of $N$. Then
		 $\rtp_{L^+}(N^+,B)\le\beth_{\omega+1}(\rtp_{\QF}(N,B))$.
	\end{enumerate}
\end{fact}

\begin{proof}  (1) is immediate as for any $\cbar,\dbar\in (N-B)^k$, $\tp(\cbar/B)=\tp(\dbar/B)$ implies $\tp_{\Delta}(\cbar/B)=\tp_{\Delta}(\dbar/B)$ implies
	$\tp_{\QF}(\cbar/B)=\tp_{\QF}(\dbar/B)$.
	
	(2)  This is Lemma 4.6 of \cite{MonNIP}.
	
	(3)  This follows from the proof of Lemma 4.7 of \cite{MonNIP}.
\end{proof}

We now state the two facts we will need about mutually algebraic theories, the first a non-structure theorem and the second a structure theorem.

\begin{fact} [{\cite[Theorem 3.2]{Expansions}}] \label{fact:exp}
	Suppose $T$ is not mutually algebraic. Then there is some expansion $T^+$ of $T$ by finitely many unary predicates and a model $N^+ \models T^+$ with a definable $X \subset N^+$ and definable $E \subset X^2$ such that $E$ is an equivalence relation with infinitely many classes, each infinite.
\end{fact}

\begin{fact} [{\cite[Propositions 4.2, 4.4]{JSL}}] \label{fact: 4.4}
	Suppose $T$ is mutually algebraic, and $M \preceq N \models T$. Then $N - M$ is partitioned into components $\set{C_i : i \in I}$ forming a forking-independent set over $M$, and such that each $C_i = \acl(c_i)\setminus M$ for any $c_i \in C_i$ and $M\cup C_i\preceq N$.
\end{fact}

From each fact we prove a corresponding lemma, which together will quickly yield our main results.

\begin{lemma} \label{lemma:unbnd}
	Suppose $T$ is a non-mutually algebraic $L$-theory, and let $\Delta$ be a reasonable set of $L$-formulas.  Then for every cardinal $\mu$, there is a cardinal $\lambda > \mu$ and models $M \prec N \models T$ with $|M| = \lambda$ and $|N| = \lambda^+$ such that for every intermediate set $M \subseteq Y \subset N$ with $|Y| = \lambda$, we have $\rtp_\Delta(N, Y) \geq \mu$.
\end{lemma}
\begin{proof}
	Consider an expansion $T^+$ of $T$ by finitely many unary predicates, $N^+ \models T^+$, and $E \subset (N^+)^2$ as in Fact \ref{fact:exp}. Fix $\lambda \geq \max(\beth_{\omega+1}(\mu), |T|)$. By possibly passing to an elementary extension, we may assume that $E$ has at least $\lambda$ classes and each $E$-class has size $\lambda^+$. By possibly adding another unary predicate, we may assume $E$ has exactly $\lambda$ classes.
	
	Let $M^+ \prec N^+$ be a Skolem hull of a transversal of $E$, so $|M^+| = \lambda$. Then for any intermediate set $M^+ \subseteq Y \subset N^+$ with $|Y| =\lambda$, both $Y$ and $N^+ - Y$ contain a point from each $E$-class, so $\rtp(N^+, Y) \geq \lambda$.
	
	Finally, we take $M, N$ to be the $L$-reducts of $M^+, N^+$. By Fact \ref{rtpfacts}, $\rtp_\Delta(N, Y) \geq \mu$.
\end{proof}

\begin{remark}
	An alternate proof of Lemma \ref{lemma:unbnd} follows from Theorem~6.1 of \cite{LT1}. One can use the infinitely many infinite arrays given by that theorem to obtain many types, in place of the infinitely many infinite $E$-classes.
\end{remark}

Before the next lemma, we introduce a doubly parameterized family of decompositions, where we vary the size of the sets using $\kappa$ and the strength of the congruence (see Definition \ref{def:decomp}) by $\Delta$, and we may also vary whether the decomposition is into subsets or elementary substructures. Pleasingly, we will see in Theorem~\ref{same} that at the level of theories, admitting essentially any of these decompositions is equivalent to mutual algebraicity.

\begin{definition} \label{def:decomp}
	Given a language $L$, fix a  set $\Delta$ of $L$-formulas and fix a cardinal $\kappa$.
	
	A {\em $\kappa$-partition} of an $L$-structure $N=A\sqcup\bigsqcup\{B_i:i\in I\}$ with $|A|\le\kappa$ and each $|B_i|\le \kappa$.
	
	A $\kappa$-partition induces an equivalence relation $\sim_{\Delta}$ on $(N\setminus A)^{<\omega}$, defined as follows.
	As notation, for $\cbar\in (N-A)^k$, if we write $\cbar=\cc_1;\dots;\cc_n$, then there are distinct $\<i_1,\dots,i_n\>$ from $I$ such that
	each $\cc_\ell\subseteq B_{i_\ell}$.  To ease notation, we write e.g., $\cc_1$ as being an initial segment of $\cbar$, although it need not be.
	
	Given $\cbar,\dbar\in (N-A)^{<\omega}$, we say $\cbar\sim_{\Delta}\dbar$ if and only if there are no repeated elements in either tuple taken individually and we can write $\cbar=\cc_1;\dots;\cc_n$  and $\dbar=\dd_1;\dots;\dd_n$ 
	with $\tp_\Delta(\cc_\ell/A)=\tp_{\Delta}(\dd_\ell/A)$ for every $1\le \ell\le n$.
	
	A $\kappa$-partition $N=A\sqcup\bigsqcup\{B_i:i\in I\}$ is a {\em $\Delta$-congruence over $A$} if,  for all $\cbar,\dbar\in (N-A)^{<\omega}$,
	$\cbar\sim_\Delta \dbar$ implies $\tp_{\Delta}(\cbar/A)=\tp_{\Delta}(\dbar/A)$.

	A {\em $(\kappa,\Delta)$-decomposition of $N$} is a $\kappa$-partition $N=A\sqcup\bigsqcup\{B_i:i\in I\}$  that is a $\Delta$-congruence over $A$,
	and a {\em $(\kappa,\Delta)$-model decomposition of $N$} is a $(\kappa,\Delta)$-decomposition of $N$ in which $A$ and each $A\cup B_i$ are universes of elementary
	substructures of $N$.
	
	For an $L$-theory $T$, we say the pair $(\kappa,\Delta)$ is {\em viable} if $\kappa\ge |T|$ and $\Delta$ is reasonable as in Definition \ref{def:reas}. 
	
	We say an $L$-theory $T$ {\em admits $(\kappa, \Delta)$-decompositions} if every $N\models T$ with $|N|\ge |T|$ has a $(\kappa,\Delta)$-decomposition, and
	$T$ {\em admits $(\kappa, \Delta)$-model decompositions} if every $N\models T$  with $|N|\ge |T|$ has a $(\kappa,\Delta)$-model decomposition.
\end{definition}

\begin{lemma} \label{lemma:ma cong}
	Let $T$ be mutually algebraic and let $M \prec N \models T$ with $|M| \leq |T|$. Let $\set{C_i : i \in I}$ be the partition of $N- M$ into components as in Fact \ref{fact: 4.4}. Then for any reasonable $\Delta$, this partition is a $(|T|, \Delta)$-model decomposition over $M$.
\end{lemma}
\begin{proof}
	By Fact~\ref{fact: 4.4}, for each $i$ we have   $M \cup C_i \preceq N$ and $|C_i|\le|T|$ since $C_i\subseteq \acl(c)$ for some singleton.
	So it remains to check that the partition is a $\Delta$-congruence over $M$. In fact, we will show the stronger statement that for any formula $\phi$, the partition is a $\phi$-congruence over $M$. This will follow from the fact that in a stable theory, if a tuple can be partitioned into two independent subtuples over a model $M$, then the $\phi$-type of the tuple over $M$ is determined by the $\phi$-type of the two independent subtuples over $M$. We write the details below.
	
	Fix a formula $\phi(\zbar)$ and tuples $\cbar, \dbar \in (N - M)^{\leq|\zbar|}$ with $\cbar \sim_\phi \dbar$. As in Definition \ref{def:decomp}, let $\cbar = \cc_1;\dots;\cc_n$  and $\dbar=\dd_1;\dots;\dd_n$ 
	with $\tp_\phi(\cc_\ell/M)=\tp_{\phi}(\dd_\ell/M)$ for every $1\le \ell\le n$. Choose a (possibly trivial) partition of $\zbar$ to give $\phi(\xbar;\ybar)$. Choose $\mbar \in M^{|\ybar|}$ and $\cc' \subseteq \cc_1\cc_2$ with $|\cc'| = |\xbar|$, and let $\cc'_1 = \cc' \cap \cc_1, \cc'_2 = \cc' \cap \cc_2$. Since $\cc_1 \ind_M \cc_2$, and forking-independence agrees with finite satisfiability over a model (since mutually algebraic theories are stable), we have $N \models \phi(\cc'_1\cc'_2, \mbar)$ if and only if there exists some $\mbar' \subset M$ such that $N\models \phi(\cc'_1\mbar', \mbar)$. Using analogous notation for $\dbar$, we have that $N \models \phi(\dd'_1\dd'_2, \mbar)$ if and only if there exists some $\mbar' \subset M$ such that $N \models \phi(\dd'_1\mbar', \mbar)$. Since $\tp_\phi(\cc_1/M) = \tp_\phi(\dd_1/M)$, this gives $N \models \phi(\cc'_1\cc'_2, \mbar) \iff N \models \phi(\dd'_1\dd'_2, \mbar)$, so $\tp_\phi(\cc_1\cc_2/M) = \tp_\phi(\dd_1\dd_2/M)$. By continuing inductively, we may show $\tp_\phi((\cc_1\cc_2)\cc_3/M) = \tp_\phi((\dd_1\dd_2)\dd_3/M)$, and eventually that $\tp_\phi(\cc/M) = \tp_\phi(\dd/M)$.
\end{proof}

Our last lemma will be useful when using decompositions to bound the number of realized types.

\begin{lemma} \label{lemma:bound}  Let $T$ be a theory and $(\kappa, \Delta)$ be viable. Let $N \models T$ and let $\{B_i:i\in I\}$ be any $(\kappa, \Delta)$-decomposition of $N$ over $A$.
	For any non-empty $J\subseteq I$ let $B_J=\bigcup_{j\in J} B_j$.
	Then for any $J \subseteq I$, $rtp_\Delta(N, AB_J)\le 2^{\kappa}$.
\end{lemma}
\begin{proof}
	
	As $|A|  \leq \kappa$, 
	there are at most $2^\kappa$ $\sim_\Delta$-classes in $(N-A)^n$ for each $n$. Thus it will suffice to show that $\cbar \sim_\Delta \dbar \Ra \tp_\Delta(\cbar/AB_J) = \tp_\Delta(\dbar/AB_J)$ for every $\cbar, \dbar \subset N \bs AB_J$. From the original congruence condition and the fact that $\cbar, \dbar$ are disjoint from $B_J$, we have $\tp_\Delta(\cbar B_J/A) = \tp_\Delta(\dbar B_J /A)$, and so $\tp_\Delta(\cbar/AB_J) = \tp_\Delta(\dbar/AB_J)$.
\end{proof}

\section{Main results}

In this section, we give some characterizations of mutual algebraicity for a theory. One is in terms of type-counting, while the others concern various types of decomposition.

\begin{theorem}  \label{same}  The following are equivalent for any theory $T$.
	\begin{enumerate}  
		\item  For some viable $(\kappa,\Delta)$, $T$ admits $(\kappa,\Delta)$-decompositions.
		\item  For all viable $(\kappa,\Delta)$, $T$ admits $(\kappa,\Delta)$-decompositions.
		\item  For some viable $(\kappa,\Delta)$, $T$ admits $(\kappa,\Delta)$-model decompositions.
		\item  For all viable $(\kappa,\Delta)$, $T$ admits $(\kappa,\Delta)$-model decompositions.
		\item  $T$ is mutually algebraic.
	\end{enumerate}
\end{theorem}

\begin{proof}  It is clear that $(1)-(3)$ follow from (4), and  $(5)\Rightarrow(4)$ is immediate from Lemma \ref{lemma:ma cong}.
	
	We now verify $(1)\Rightarrow(5)$.  By way of contradiction, suppose there is some viable $(\kappa, \Delta)$
	such that $T$ admits $(\kappa,\Delta)$-decompositions, but $T$ is not mutually algebraic.  
	
	Let $\mu > 2^\kappa$, let $M \prec N \models T$ and $\lambda > \mu$ be as in Lemma \ref{lemma:unbnd}, and let $A \sqcup \bigsqcup \set{B_i : i \in I}$ be a $(\kappa, \Delta)$-decomposition of $N$. Let $J \subset I$ be minimal such that $AB_J$, in the notation of Lemma \ref{lemma:bound}, covers $M$. Then $M \subseteq AB_J$ and $|AB_J| = \lambda$, so $\rtp_\Delta(N, AB_J) \geq \mu > 2^\kappa$ by Lemma \ref{lemma:unbnd}. But this contradicts Lemma \ref{lemma:bound}.
\end{proof}

In proving $(1) \Ra (5)$ in Theorem \ref{same}, there is a tension between taking $\Delta = \text{QF}$ and $\Delta = \text{FO}$. On the one hand, our non-structure result for non-mutually algebraic theories yields an FO-definable equivalence relation in a unary expansion. However, although it is easy that taking a unary expansion preserves admitting $(\kappa, \text{QF})$-congruences, this is not clear for $(\kappa, \text{FO})$-congruences, which prevents pulling the non-structure back to the original theory. By instead passing through type-counting, Lemma \ref{lemma:bound} allows us to relate $\Delta = \text{QF}$ and $\Delta = \text{FO}$. We now also characterize mutual algebraicity in terms of this sort of type counting.

\begin{definition}  Call a (possibly incomplete) theory $T$ {\em bounded} if there is some cardinal $\kappa$ such that for any $M \models T$ (of any size), there are at most $\kappa$ coordinate-wise non-algebraic types over $M$. Equivalently, $\rtp(N,M)\le \kappa$ for all
	$M\preceq N\models T$.
\end{definition}

The notion of a theory being bounded was investigated in \cite[Corollary 6.1.8]{BS}, which proves that $T$ is bounded if and only if it is strongly decomposable (i.e. admits ($|T|$, QF)-model decompositions). 

\begin{theorem} \label{thm:bnd}  
	A theory $T$ is mutually algebraic if and only if $T$ is bounded. Furthermore, if $T$ is bounded then it is bounded by $2^{|T|}$, and if $T$ is not bounded then it is not bounded even for quantifier-free types.
\end{theorem}

\begin{proof}  First, assume $T$ is mutually algebraic and let $M \preceq N \models T$. Let $M_0 \preceq M$ with $|M_0| \leq |T|$, and consider the partition of $N$ over $M_0$ into components $\set{C_i : i \in I}$ as in Fact \ref{fact: 4.4}. By Lemma \ref{lemma:ma cong}, this is a $(|T|, \text{FO})$-decomposition of $N$ over $M_0$. Since $M$ is algebraically closed, we have $M = M_0C_J$, in the notation of Lemma \ref{lemma:bound},  for some $J \subset I$. Thus by Lemma \ref{lemma:bound}, $\rtp(N, M) \leq 2^{|T|}$.
	
	Conversely, if $T$ is not mutually algebraic, the statement holds by Lemma \ref{lemma:unbnd}.
\end{proof}

\begin{remark}  \leavevmode
		\begin{enumerate}
			\item  In the definition of boundedness, it is crucial that the base be restricted to elementary submodels of $N$.
			As an example, take $L=\{R\}$ and let $N$ be an infinite model of `mated pairs,' i.e., $R$ is symmetric, irreflexive, and
			every element of $N$ is $R$-related to exactly one element.  Then $Th(N)$ is mutually algebraic and totally categorical.  But, for any infinite cardinal $\lambda$,
			taking $N$ to be the model of size $\lambda$ and $B$ to be a set of $R$-representatives, we have $\rtp(N,B)=\rtp_\QF(N,B)=\lambda$.
			
			\item  The bound of $2^{|T|}$ in Theorem~\ref{thm:bnd} is sharp, as witnessed by the theory $T$ of $\kappa$ independent unary predicates.
			Then $|T|=\kappa$ and is mutually algebraic.  However, if $N\models T$ realizes all of the $2^\kappa$ types over $\emptyset$ and if $M\preceq N$ is any elementary
			substructure of size $<2^\kappa$, then $\rtp(N,M)=\rtp_{\QF}(N,M)=2^\kappa$.
			\item Theorem \ref{thm:bnd} is similar to the main result of \cite{LT1}, which, for a finite relational language, characterizes mutual algebraicity by there being only finitely many quantifier-free coordinate-wise non-algebraic $n$-types for each $n$, over every model.
		\end{enumerate}
\end{remark}

We close with a question.  Even though many notions of decompositions mentioned in Theorem \ref{same} are all equivalent to mutual algebraicity at the level of theories, requiring that the base set $A=\emptyset$ 
is more restrictive.  
That is, define an $\emptyset$-$(\kappa,\Delta)$-decomposition of $N$ to be a $(\kappa,\Delta)$-decomposition of $N$ in which $A=\emptyset$.  As an easy example, take $L=\{E\}$
and let $T$ be the complete $L$-theory asserting that $E$ is an equivalence relation with two classes, both infinite.  Then $T$ is mutually algebraic, but if $N$ is the saturated model of
size $\aleph_1$, then $N$ does not have an $\emptyset$-$(\aleph_0,\QF)$-decomposition since there is only one 1-type over the empty set.
It would be desirable to characterize those (mutually algebraic) theories that admit $\emptyset$-$(\kappa,\Delta)$-decompositions.

\bibliographystyle{amsplain}
\bibliography{Bib.bib}

\end{document}